\documentclass[11pt, reqno]{amsart}
\usepackage{graphicx,indentfirst}
\usepackage{amscd}

\usepackage[all]{xy}
\usepackage{amsmath}
\usepackage{amsfonts,amssymb}
\usepackage{amsthm}
\usepackage{latexsym}
\usepackage[american]{babel}
\usepackage{mathrsfs}
\usepackage{times}
\usepackage{tikz}
\usepackage{verbatim}
\usetikzlibrary{matrix,arrows,decorations.pathmorphing}
\usepackage[a4paper,left=2.5cm,right=2.5cm,bottom=3cm,top=3.5cm]{geometry}
\RequirePackage[T1]{fontenc}

\begin{document}

\title{The full exceptional collections of categorical resolutions of curves}
\author{Zhaoting Wei}
\address{Department of Mathematics, 831 E 3rd Street, Indiana University, Bloomington, IN, 47405}
\email{zhaotwei@indiana.edu}

\newcommand{\coh}{\text{coh}}
\newcommand{\perf}{\text{perf}}
\newcommand{\Hom}{\text{Hom}}
\newcommand{\Pic}{\text{Pic}}
\newcommand{\red}{\text{red}}
\newcommand{\Spec}{\text{Spec}}
\newcommand{\Frac}{\text{Frac}}
\newcommand{\rD}{\mathrm{D}}

\newtheorem{thm}{Theorem}[section]
\newtheorem{lemma}[thm]{Lemma}
\newtheorem{prop}[thm]{Proposition}
\newtheorem{coro}[thm]{Corollary}
\theoremstyle{definition}\newtheorem{defi}{Definition}[section]
\theoremstyle{remark}\newtheorem{eg}{Example}
\theoremstyle{remark}\newtheorem{rmk}{Remark}

\maketitle


\begin{abstract}
This paper gives a complete answer of the following question: which (singular, projective) curves have a categorical resolution of singularities which admits a full exceptional collection? We   prove that such full exceptional collection exists if and only if the geometric genus of the curve equals to $0$. Moreover we can also prove that a curve with geometric genus equal or greater than $1$ cannot have  a categorical resolution of singularities which has a tilting object. The proofs of both results are given by a careful study of the Grothendieck group and the Picard group of that curve.
\end{abstract}

\small{Keywords: categorical resolution, singular curve, K-theory, exceptional collection, tilting object

}

\section{Introduction}\label{Introduction}
For a triangulated category $\mathcal{C}$, having a full exceptional collection is a very good property.
Recall that the definition of full exceptional collection is as follows.

\begin{defi}\label{defi:full excep coll}
A full exceptional collection of a
triangulated category $\mathcal{C}$ is a collection $\{A_1\ldots A_n\}$ of objects such that
\begin{enumerate}
\item for all $i$ one has $\Hom_{\mathcal{C}}(A_i,A_i)=k$ and $\Hom_{\mathcal{C}}(A_i,A_i[l])=0$ for all $l\neq 0$;
\item for all $1\leq i<j\leq n$ one has $\Hom_{\mathcal{C}}(A_j,A_i[l])=0$ for all $l\in \mathbb{Z}$;
\item  the smallest triangulated subcategory of $\mathcal{C}$  containing
$A_1,\ldots, A_n$ coincides with $\mathcal{C}$.
\end{enumerate}
\end{defi}

However it is not very common that a triangulated category $\mathcal{C}$ has a full exceptional collection. In algebraic geometry, it is well-known that for a smooth projective curve $X$ over an algebraically closed field $k$, its bounded derived category of coherent sheaves $\rD^b(\coh(X))$  has a full exceptional collection if and only if the genus of $X$ equals to $0$.

Moreover for a singular projective curve $X$ and a (geometric) resolution of singularities $\widetilde{X}\to X$, the geometric genus of $\widetilde{X}$ and $X$ are equal, hence it is clear that $\rD^b(\coh(\widetilde{X}))$ has a full exceptional collection if and only if the geometric genus of $X$ equals to $0$.

In this paper we would like to consider the categorical resolution of $X$, which is introduced in \cite{kuznetsov2008lefschetz}.

\begin{defi}\label{defi: categorical resolution}[\cite{kuznetsov2008lefschetz} Definition 3.2 or \cite{kuznetsov2014categorical} Definition 1.3]
A categorical resolution of a scheme $X$ is a smooth, cocomplete,
compactly generated, triangulated category $\mathscr{T}$ with an adjoint pair of triangulated
functors
$$
\pi^*: \rD(X)\to \mathscr{T} \text{ and } \pi_*: \mathscr{T}\to \rD(X)
$$
such that
\begin{enumerate}
\item $\pi_*\circ \pi^*=id$;
\item both $\pi_*$ and $\pi^*$ commute with arbitrary direct sums;
\item $\pi_*(\mathscr{T}^c)\subset \rD^b(\coh(X))$ where $\mathscr{T}^c$ denotes the full subcategory of $\mathscr{T}$ which consists of compact objects.
\end{enumerate}
\end{defi}

\begin{rmk}
The first property implies that $\pi^*$ is fully faithful and the second property implies that $\pi^*(\rD^{\perf}(X))\subset \mathscr{T}^c$.
\end{rmk}

\begin{rmk}
The categorical resolution of $X$ is not necessarily unique.
\end{rmk}

\begin{rmk}
In this paper we will not discuss further on the smoothness of a triangulated category and the interested readers may refer to \cite{kuznetsov2014categorical} Section 1. Moreover, the main result in this paper does not depend on the smoothness, see Corollary \ref{coro: non-exist of full excep collection for factor through reduced case} and Corollary \ref{coro: non-exist of full excep collection for factor through general case} below.
\end{rmk}

We are interested in the question that when does $\mathscr{T}^c$ have a full exceptional collection. If $X$ is an   projective curve of geometric genus $g=0$, it can be deduced from the construction in \cite{kuznetsov2014categorical} that there exists a categorical resolution $(\mathscr{T},\pi^*,\pi_*)$ of $X$ such that $\mathscr{T}^c$ has a full exceptional collection. See Proposition \ref{prop: g=0 has a full exc coll} below.

The main result of this paper is the following theorem, which rules out the possibility for any categorical resolution of a curve with geometric genus $g \geq 1$ has a full exceptional collection.

\begin{thm}\label{thm: non-exist of full excep collection in the introduction}[See Theorem \ref{thm: non-exist of full excep collection general case} below]
Let $X$ be a  projective curve over an algebraically closed field $k$. Let $(\mathscr{T},\pi^*,\pi_*)$ be a categorical resolution of $X$. If the geometric genus of $X$ is $\geq 1$, then $\mathscr{T}^c$ cannot have a full exceptional collection.

In other words, $X$ has a categorical resolution which admits a full exceptional collection if and only if the geometric genus of $X$ equals to $0$.
\end{thm}

\begin{rmk}\label{rmk: extist of f.d. algebra for g=0}
In a recent paper \cite{burban2015singular} a result which is related to the above claim has been proved. Actually it has been proved that if $X$ is a reduced rational curve, then there exists a categorical resolution $(\mathscr{T},\pi^*,\pi_*)$ of $X$ such that $\mathscr{T}^c$ has a tilting object, which in general does not come from an exceptional collection. See \cite{burban2015singular} Theorem 7.4.
\end{rmk}

Recall that the definition of tilting object is given as follows.

\begin{defi}\label{defi: tilting object}
Let $\mathcal{C}$ be a triangulated category. A tilting object is an object $L$ of $\mathcal{C}$ which satisfies the following properties.
\begin{enumerate}
\item $L$ is a compact object of $\mathcal{C}$;
\item $\Hom_{\mathcal{C}}(L,L[i])=0$ for any non-zero integer $i$;
\item the smallest thick triangulated subcategory of $\mathcal{C}$ which contains $L$ is $\mathcal{C}$ itself.
\end{enumerate}

For a tilting object let $\Lambda=\text{End}_{\mathcal{C}}(L)$. Then it can be shown that we have equivalence of triangulated categories
$$\mathcal{C}\cong \rD^b(\Lambda-\text{mod})$$
where $\rD^b(\Lambda-\text{mod})$ is the derived category of bounded complexes of finitely generated $\Lambda$-modules.
\end{defi}

Actually we can also prove a related result in the $g\geq 1$ case. (thanks to Igor Burban for pointing it out)

\begin{thm}\label{thm: non-exist of f.d algebra as a cat resolution in the introduction} [See Theorem \ref{thm: non-exist of f.d algebra as a cat resolution} below]
Let $X$ be a projective curve over an algebraically closed field $k$ of geometric genus $\geq 1$. Let $(\mathscr{T},\pi^*,\pi_*)$ be a categorical resolution of $X$. Then $\mathscr{T}^c$ cannot have a tilting object, moreover there cannot be a finite dimensional $k$-algebra $\Lambda$ of finite global dimension such that
$$
\mathscr{T}^c\cong \rD^b(\Lambda-\text{mod})
$$
\end{thm}

The proofs of both theorems depend on a careful study of various Grothendieck groups of $X$. In particular we will investigate the natural map $K_0(\rD^{\perf}(X))\to K_0(\rD^b(\coh(X)))$ and show that if $g\geq 1$ then the image is not finitely generated, of which Theorem \ref{thm: non-exist of full excep collection in the introduction} and \ref{thm: non-exist of f.d algebra as a cat resolution in the introduction} will be a direct consequence.

\section{Some generalities on K-theory and the Picard group}
In this section we quickly review the K-theory and the Picard group of schemes. For reference see \cite{weibel2013k} Chapter II.

Let $\mathcal{A}$ be an abelian category (or more generally an exact category). The Grothendieck group $K_0(\mathcal{A})$ is defined as an abelian group with generators $[A]$ for each isomorphism class of objects $A$ in $\mathcal{A}$ and subjects to the relation that
$$
[A_2]=[A_1]+[A_3]
$$
for any short exact sequence $0\to A_1\to A_2\to A_3\to 0$ in $\mathcal{A}$.

Similarly let $\mathcal{C}$ be a triangulated category. The Grothendieck group $K_0(\mathcal{C})$ is defined as an abelian group with generators $[C]$ for each isomorphism class of objects $C$ in $\mathcal{C}$ and subjects to the relation that
$$
[C_2]=[C_1]+[C_3]
$$
for any exact triangle $C_1\to C_2\to C_3\to C_1[1]$ in $\mathcal{C}$.

\begin{prop}\label{prop: Grothen group of full excep coll}
If a triangulated category $\mathcal{C}$ has a full exceptional collection $\{A_1\ldots A_n\}$, then the Grothendieck group of $\mathcal{C}$, $\mathrm{K}_0(\mathcal{C})$, is isomorphic to $\mathbb{Z}^n$.
\end{prop}
\begin{proof}
It is an immediate consequence of Definition \ref{defi:full excep coll}.
\end{proof}

\begin{defi}\label{defi: groth gps of perf and coh}
Let $X$ be an Noetherian scheme, follow the standard notation (see for example \cite{srinivas1996algebraic} Section 5.6 or \cite{weibel2013k} Chapter II) we denote the Grothendieck group of $\rD^{\perf}(X)$ by $K_0(X)$ and the Grothendieck group of $\rD^b(\coh(X))$ by $G_0(X)$.

Notice that in some literatures, say \cite{berthelot1966seminaire} Expos\'{e} IV or \cite{manin1969lectures}, $K_0(X)$ is denoted by $K^0(X)$ and $G_0(X)$ is denoted by $K_0(X)$. Nevertheless in this paper we will use the previous  notation.
\end{defi}

\begin{rmk}\label{rmk: K naive and K}
In the literature people also define $K^{\text{na\"{i}ve}}_0(X)$ to be the Grothendieck group of the exact category $VB(X)$ and $G^{\text{na\"{i}ve}}_0(X)$ to be the Grothendieck group of the abelian category $\coh(X)$.

Nevertheless $G^{\text{na\"{i}ve}}_0(X)$ is isomorphic to $G_0(X)$ for any Noetherian scheme $X$ (\cite{berthelot1966seminaire}, Expos\'{e} IV, 2.4) and $K^{\text{na\"{i}ve}}_0(X)$ is isomorphic to $K_0(X)$ for any quasi-projective scheme $X$ (\cite{berthelot1966seminaire}, Expos\'{e} IV, 2.9). Since we always work with quasi-projective schemes in this paper, we can identify $G^{\text{na\"{i}ve}}_0(X)$ and $G_0(X)$ as well as $K^{\text{na\"{i}ve}}_0(X)$ and $K_0(X)$.
\end{rmk}

\begin{defi}\label{defi: cartan homomorphism}
Let $X$ be a Noetherian scheme. The inclusion $\rD^{\perf}(X)\hookrightarrow \rD^b(\coh(X))$  gives a group homomorphism
$$
c: K_0(X)\to G_0(X)
$$
which is called the Cartan homomorphism.
\end{defi}

\begin{prop}\label{prop: Cartan map is a module map}
For a Noetherian scheme $X$, the tensor product gives $K_0(X)$   a ring structure and $G_0(X)$   a   $K_0(X)$-module structure. Moreover, the Cartan homomorphism $c: K_0(X)\to G_0(X)$ is a  morphism of $K_0(X)$-modules.
\end{prop}
\begin{proof}
See \cite{manin1969lectures} 1.5 and 1.6.
\end{proof}

\begin{prop}\label{prop: K=G for regular schemes}
If $X$ is a regular Noetherian scheme, then the Cartan homomorphism is an isomorphism, i.e. we have
$$
c: K_0(X)\overset{\cong}{\to} G_0(X)
$$
\end{prop}
\begin{proof}
See \cite{weibel2013k} Chapter II Theorem 8.2.
\end{proof}

Smooth schemes are regular hence the Cartan homomorphism is an isomorphism for any smooth scheme.

\begin{rmk}\label{rmk: Cartan is not an isom in general}
For general $X$ the Cartan homomorphism is not an isomorphism, actually it is not even injective in general.
\end{rmk}

Next we talk about the functorial properties of $K_0$ and $G_0$, which  are more involved. First we have the following definition.

\begin{defi}\label{defi: pull back of K_0 and G_0}
Let $f: X\to Y$ be a morphism of schemes, then the derived pull-back $Lf^*$ functor induces the map
$$
f^*: K_0(Y)\to K_0(X).
$$
See \cite{berthelot1966seminaire} Expos\'{e} IV, 2.7.

If $f: X\to Y$ is a flat morphism between Noetherian schemes, or more generally $f$ is of finite Tor-dimension. Then $Lf^*$ is a functor $D^b(\coh(Y))\to D^b(\coh(X))$ and induces the map
$$
f^*: G_0(Y)\to G_0(X).
$$
See \cite{berthelot1966seminaire} Expos\'{e} IV, 2.12.
\end{defi}

We can also define the push-forward map for $G_0(-)$ for proper morphisms.

\begin{defi}\label{defi: push forward for G_0}
Let $f: X\to Y$ be a proper morphism of Noetherian schemes, then the derived push-forward functor $Rf_*$ induces the map
$$
f_*: G_0(X)\to G_0(Y).
$$
\end{defi}

We will also need some results on the relationship between the Grothendieck group and the Picard group. Let $\Pic(X)$ denote the Picard group of $X$ and we have the following proposition.

\begin{prop}\label{prop: K_0 to Pic det}
There is a determinant map
$$
\det: K_0(X)\to \Pic(X)
$$
which is a surjective group homomorphism. Moreover, the determinant map commutes with the restriction map, i.e. we have the following commutative diagram
$$
\begin{CD}
K_0(X) @>\det>> \Pic(X)\\
@VVrV @VVrV\\
K_0(U) @>\det>> \Pic(U)
\end{CD}
$$
\end{prop}
\begin{proof}
For an $n$-dimensional vector bundle $\mathcal{E}$ we can take its determinant line bundle, i.e. the top exterior power $\wedge^n \mathcal{E}$ and we call it $\det(\mathcal{E})$. Moreover, for a short exact sequence of vector bundles $0\to \mathcal{E}\to \mathcal{F}\to \mathcal{G}\to 0$ we have $\det(\mathcal{F})\cong \det(\mathcal{E})\otimes \det(\mathcal{G})$ hence we get a well-defined group homomorphism $\det: K_0(X)\to \Pic(X)$.

The above diagram commutes because the construction of the determinant map is natural. The surjectivity of   det also comes from the construction since we could pick $\mathcal{E}$ to be any line bundle and hence $\det(\mathcal{E})=\mathcal{E}$.
\end{proof}

\section{The irreducible and reduced case of the main theorem}\label{section: the reduced case}
To illustrate the idea,  we focus on the case that $X$ is an irreducible, reduced, projective curve over $k$ in this section.

In this case let $p: \widetilde{X}\to X$ be a (geometric) resolution of singularity and we  can obtain more information  on $\Pic(\widetilde{X})$. First we have

\begin{thm}[\cite{liu2002algebraic} Corollary 7.4.41]\label{thm: picard group of curves genus >1}
Let $\widetilde{X}$ be a smooth, connected, projective curve over an algebraically closed field $k$, of genus $g$. Let $\Pic^0(\widetilde{X})$ denote the subgroup of $\Pic(\widetilde{X})$ consisting of divisors of degree $0$. Let $n\in \mathbb{Z}$ be non-zero and $\Pic^0(\widetilde{X})[n]$ denote the kernel of the multiplication by $n$ map.
\begin{enumerate}
\item If $(n,\text{char} (k))=1$, then $\Pic^0(\widetilde{X})[n]\cong (\mathbb{Z}/n\mathbb{Z})^{2g}$;
\item If $p=\text{char} (k)>0$, then there exists an $0\leq h\leq g$ such that for any $n=p^m$, we have $\Pic^0(\widetilde{X})[n]=(\mathbb{Z}/n\mathbb{Z})^h$.
\end{enumerate}
\end{thm}

\begin{coro}\label{coro: Picard is infinitely generated}
Let $\widetilde{X}$ be a smooth, connected, projective curve over an algebraically closed field $k$ of genus $g\geq 1$, then $\Pic^0(\widetilde{X})$ and hence $\Pic(\widetilde{X})$ are not finitely generated as an abelian group. Moreover, for any non-zero integer $n$, $n\Pic(\widetilde{X})$ is not finitely generated.
\end{coro}
\begin{proof}
It is  an immediate consequence of  Theorem \ref{thm: picard group of curves genus >1}.
\end{proof}

\begin{rmk}\label{rmk: picard group non-algebraic closed case}
If the base field $k$ is not algebraically closed, then $\Pic^0(\widetilde{X})$ may be finitely generated. For example if $k=\mathbb{Q}$ and $\widetilde{X}$ is a smooth elliptic curve, then by Mordell theorem, $\Pic^0(\widetilde{X})$ is a finitely generated abelian group.
\end{rmk}

Let $Z$ be the closed subset consisting of singular points of $X$ and $U=X-Z$. Since $p: \widetilde{X}\to X$ is a resolution of singularity,  the restriction of $p$
$$
p|_{p^{-1}(U)}: p^{-1}(U)\overset{\cong}{\to} U
$$
is an isomorphism.

We want to understand the picard group of $U$. In fact we have the following result

\begin{lemma}\label{lemma: Picard gp of U is infi generated}
Let $\widetilde{X}$ be a smooth and connected projective curve with genus $g\geq 1$ over an algebraically closed field $k$. Let $U$ be a non-empty open subset of $\widetilde{X}$. Then $\Pic(U)$ is not finitely generated. Moreover, for any non-zero integer $n$, $n\Pic(U)$ is not finitely generated.
\end{lemma}
\begin{proof}
This is actually part of \cite{liu2002algebraic} Exercise 7.4.9. Thanks to Georges Elencwajg for helping with the proof. Actually we can write $U=X\backslash \{p_1,\ldots ,p_l\}$. It follows that the kernel of the natural homomorphism $\Pic^0(X)\to \Pic(U)$ is   the subgroup of $\Pic^0(X)$ generated by $[p_i]-[p_j]$, hence is finitely generated. Then this lemma is a consequence of Corollary \ref{coro: Picard is infinitely generated}.
\end{proof}

It is also necessary to know the relation between the Picard group of a non-smooth curve $X$ and its non-empty subscheme $U$, which is given in the following lemma.

\begin{lemma}\label{lemma: ext of line bundles on singular curves}
Let $X$ be a (not necessarily smooth) curve over an algebraically closed field $k$. Let  $U$ be an open subscheme of $X$.

Let $\mathcal{L}$ be a line bundle on $U$. Then we can always extend $\mathcal{L}$ to a line bundle on $X$. As a result, the restriction map of the Picard groups
$$
r: \Pic(X)\to \Pic(U)
$$
is surjective
\end{lemma}
\begin{proof}
One way to proof this result (thanks to K\c{e}stutis \v{C}esnavi\v{c}ius for pointing it out) is to first find a Cartier divisor $D$ on $U$ whose associated line bundle is $\mathcal{L}$. The existence of such $D$ is guaranteed by \cite{grothendieck1967elements} Proposition 21.3.4 (a). Then apply \cite{grothendieck1967elements} Proposition 21.9.4 we can extend $D$ to a Cartier divisor $D^{\prime}$ on $X$, whose associated line bundle $\mathcal{L}^{\prime}$ gives an extension of $\mathcal{L}$.
\end{proof}

The next Proposition is the key step of our proof.

\begin{prop}\label{prop: Cartan homo has infinite image, reduced case}
Let $X$ be a reduced, irreducible, projective curve of geometric genus $g\geq 1$ over an algebraically closed field $k$, then the image of the Cartan homomorphism
$$
c: K_0(X)\to G_0(X)
$$
is not finitely generated.
\end{prop}
\begin{proof}
First let $Z$ be the closed subset consisting of singular points of $X$ and $U=X-Z$ be the smooth open subscheme.  We have the restriction maps $r: K_0(X)\to K_0(U)$ and $r: G_0(X)\to G_0(U)$ and they give the  commutative diagram
$$
\begin{CD}
K_0(X) @>c>> G_0(X)\\
@VVrV @VVrV\\
K_0(U) @>c>> G_0(U)
\end{CD}
$$
Since $U$ is smooth, by Proposition \ref{prop: K=G for regular schemes} the bottom map is an isomorphism.

Now assume the image of the top map is finitely generated, then the image of the composition $r\circ c: K_0(X)\to G_0(U)$ is also finitely generated. Since we have the isomorphism $c: K_0(U)\overset{\cong}{\to}G_0(U)$, the left vertical map $r: K_0(X)\to K_0(U)$ must  also have  finitely generated image. Therefore the image of the composition
$$
K_0(X)\overset{r}{\to} K_0(U)\overset{\det}{\to}\Pic(U)
$$
is  finitely generated.

On the other hand we consider the commutative diagram
$$
\begin{CD}
K_0(X) @>\det>> \Pic(X)\\
@VVrV @VVrV\\
K_0(U) @>\det>> \Pic(U)
\end{CD}
$$
By Proposition \ref{prop: K_0 to Pic det} and Lemma \ref{lemma: ext of line bundles on singular curves}, the top and the right vertical map of the above diagram are surjective and so does their composition. As a result $\Pic(U)=\Pic(p^{-1}(U))$ is finitely generated, which is contradictory to Lemma \ref{lemma: Picard gp of U is infi generated}.
\end{proof}

\begin{coro}\label{coro: non-exist of full excep collection for factor through reduced case}
Let $X$ be a reduced, irreducible, projective curves of geometric genus $g\geq 1$ over an algebraically closed field $k$. If the inclusion $\rD^{\perf}(X)\to \rD^b(\coh(X))$ factors through a triangulated category $\mathcal{S}$, then $\mathcal{S}$ cannot have a full exceptional collection.
\end{coro}
\begin{proof}
The composition
$$
K_0(X)\to K_0(\mathcal{S})\to G_0(X)
$$
coincides with the Cartan homomorphism $c: K_0(X)\to  G_0(X)$. By Proposition \ref{prop: Cartan homo has infinite image, reduced case}, the image of the Cartan homomorphism is not finitely generated, hence $ K_0(\mathcal{S})$ is not finitely generated. Then by Proposition \ref{prop: Grothen group of full excep coll}, $\mathcal{S}$ cannot have a full exceptional collection.
\end{proof}

\begin{coro}\label{coro: non-exist of full excep collection reduced case}
Let $X$ be a reduced, irreducible, projective curves of geometric genus $g\geq 1$ over an algebraically closed field $k$. Let  $(\mathscr{T},\pi_*,\pi^*)$ be a categorial resolution of $X$. Then  $\mathscr{T}^c$ cannot have a full exceptional collection.
\end{coro}
\begin{proof}
By the definition of categorical resolution, the composition
$$
\rD^{\perf}(X)\overset{\pi^*}{\to}\mathscr{T}^c \overset{\pi_*}{\to}\rD^b(\coh(X))
$$
is the same as the inclusion $\rD^{\perf}(X)\hookrightarrow \rD^b(\coh(X))$. Therefore the composition
$$
K_0(X)\to K_0(\mathscr{T}^c)\to G_0(X)
$$
coincides with the Cartan homomorphism $c: K_0(X)\to  G_0(X)$. Then it is a direct consequence of Corollary \ref{coro: non-exist of full excep collection for factor through reduced case}.
\end{proof}


\section{The general case of the main theorem}\label{section: the general case}
In this section we consider the case that $X$ is not irreducible nor reduced. In this case we still want to show that the image of the Cartan homomorphism $c: K_0(X)\to G_0(X)$ is not finitely generated but the proof is more involved.

Let $X_{\red}$ denote the associated reduced scheme of $X$ and $i:X_{\red}\to X$ the natural closed immersion. Then $X_{\red}$ is a  reduced, projective curve with the same geometric genus as $X$.

First we investigate the $g=0$ case, which is the following Proposition.

\begin{prop}\label{prop: g=0 has a full exc coll}
Let $X$ be a   projective curve over an algebraically closed field $k$ of geometric genus $g=0$, then $X$ has a categorical resolution $(\mathscr{T},\pi^*,\pi_*)$ such that $\mathscr{T}^c$  has a full exceptional collection.
\end{prop}
\begin{proof}
As we mentioned in the Introduction, the result in this Proposition is a direct consequence of the construction of categorical resolution in \cite{kuznetsov2014categorical}, although it is not explicitly stated in \cite{kuznetsov2014categorical}.

First \cite{kuznetsov2014categorical} Equation (59) in page 69 gives a chain
\begin{equation}
\xymatrix{X_m  \ar[r] &X_{m-1}\ar[r]&\ldots \ar[r]&X_1\ar[r]&X_0\ar@{=}[r]&X\\
& Z_{m-1}\ar@{^{(}->}[u] & &Z_1\ar@{^{(}->}[u]&Z_0\ar@{^{(}->}[u]}
\end{equation}
where each $X_{i+1}$ is the blowup of $X_i$ at the center $Z_i$ and $(X_m)_{\red}$ is smooth.

Moreover \cite{kuznetsov2014categorical} Equation (61) in page 71 tells us that there exists a categorical resolution $\mathscr{T}$ of $X$ such that its   subcategory $\mathscr{T}^c$ has the following semiorthogonal decomposition
\begin{equation}
\begin{split}
\mathscr{T}^c=&\langle\underbrace{\rD^b(\coh(Z_0))\ldots \rD^b(\coh(Z_0))}_{n_0 \text{ times}},\ldots,\\
&\underbrace{\rD^b(\coh(Z_{m-1}))\ldots \rD^b(\coh(Z_{m-1}))}_{n_{m-1} \text{ times}},\\
&\underbrace{\rD^b(\coh((X_m)_{\red}))\ldots \rD^b(\coh((X_m)_{\red}))}_{n_m \text{ times}}\rangle
\end{split}
\end{equation}
where the $n_i$'s are certain multiples given in \cite{kuznetsov2014categorical} after Equation (61) and we do not need their precise definition.

Since $X$ is of dimension $1$, each of the $Z_i$ is $0$-dimensional hence $\rD^b(\coh(Z_i))$ has a full exceptional collection. Moreover since $X$ is of genus $0$, we have $(X_m)_{\red}$ is a finite product of $\mathbb{P}^1$'s hence $\rD^b(\coh((X_m)_{\red}))$ also has a full exceptional collection. As a result $\mathscr{T}^c$ has a full exceptional collection.
\end{proof}

Then we consider the $g\geq 1$ case. By Definition \ref{defi: pull back of K_0 and G_0} and \ref{defi: push forward for G_0} we have the natural map
$$
i^*:K_0(X)\to K_0(X_{\red})
$$
and
$$
i_*:G_0(X_{\red})\to G_0(X).
$$

For $i_*$ we have the following "devissage" theorem.

\begin{thm}\label{thm: devissage}[\cite{weibel2013k} Chapter II Corollary 6.3.2]
Let $X$ be a Noetherian scheme, and $X_{\red}$ the associated reduced
scheme. Then $
i_*:G_0(X_{\red})\to G_0(X)
$ is an isomorphism.
\end{thm}
\begin{proof}
See \cite{weibel2013k} Chapter II Corollary 6.3.2.
\end{proof}

However, the following diagram
$$
\begin{CD}
K_0(X) @>c>> G_0(X)\\
@VVi^*V @A\cong Ai_*A\\
K_0(X_{\red}) @>c>> G_0(X_{\red})
\end{CD}
$$
does not commute. Hence we cannot directly apply the result in Section \ref{section: the reduced case} and need to find another way.

Let $X=\cup_{i=1}^m X_i$ be the decomposition into irreducible components, hence $X_{\red}=\cup_{i=1}^m (X_i)_{\red}$ (Do not confused with the $X_i$'s in the proof of Proposition \ref{prop: g=0 has a full exc coll}). Since $X$ has geometric genus $\geq 1$, at least one of the irreducible components $X_i$'s also has geometric genus $\geq 1$, say $X_1$.

For an non-empty, open, irreducible subscheme $U$ of $X_1$ we also consider $U_{\red}$. We can make $U$ small enough so that both $U$ and $U_{\red}$ are affine and $U_{\red}$  is smooth. Let $U=\Spec(A)$ and $U_{\red}=\Spec(A/I)$ where $I$ is the nilpotent radical of $A$ with $I^{l+1}=0$. Since $U$ is irreducible, $I$ is also the minimal prime ideal of $A$. Let $\mathcal{I}$ denote the associated sheaf on $U$.

Let us consider the diagram
$$
\begin{CD}
K_0(U) @>c>> G_0(U)\\
@VVi^*V @AAi_*A\\
K_0(U_{\red}) @>c>> G_0(U_{\red})
\end{CD}
$$
Again it does not commute. Nevertheless we will prove that it is not too far from commutative.

First let us fix the notations. Let $e_U$ denote the element $[\mathcal{O}_U]$ in $G_0(U)$ and $e_{U_{\red}}$ denote the element $[\mathcal{O}_{U_{\red}}]$ in $G_0(U_{\red})$.

\begin{lemma}\label{lemma: n times reduced in G_0}
We can choose $U$ small enough such that there is a non-zero integer $n$ such that
$$
e_U=n\,i_*(e_{U_{\red}}).
$$
\end{lemma}
\begin{proof}
By Theorem \ref{thm: devissage}, $i_*$ is an isomorphism so it is sufficient to prove
$$
i_*^{-1}(e_U)=n \,e_{U_{\red}}
$$
 in   $G_0(U_{\red})$.

It is clear that in $G_0(U_{\red})$ we have
\begin{equation}\label{equa: decompose of e_U}
i_*^{-1}(e_U)=e_{U_{\red}}+[\mathcal{I}/\mathcal{I}^2]+\ldots+[\mathcal{I}^{l-1}/\mathcal{I}^l]+[\mathcal{I}^l].
\end{equation}
Each of the $\mathcal{I}^{k-1}/\mathcal{I}^k$ is a coherent sheaf on the smooth scheme $U_{\red}$ hence we have a resolution of finite length
$$
0\to \mathcal{P}^{m_k}_k\to \mathcal{P}^{m_k-1}_k\to \ldots \to  \mathcal{P}^{0}_k\to \mathcal{I}^{k-1}/\mathcal{I}^k ~\text{ for }1\leq k\leq l+1.
$$
where the $\mathcal{P}^{m_k-j}_k$'s are locally free sheaves on $U_{\red}$. We can shrink $U$ further to make all the $\mathcal{P}^{m_k-j}_k$'s are   free sheaves on $U_{\red}$. Hence for each $k$ there is an integer $n_k$ such that
$$
[\mathcal{I}^{k-1}/\mathcal{I}^k]=n_k e_{U_{\red}}
$$
and as a result there is an integer $n$ such that
$$
i_*^{-1}(e_U)=n \,e_{U_{\red}}
$$
 in   $G_0(U_{\red})$.

 We still need to show that $n \neq 0$. This can be achieved by localizing to the generic point of $U$. Recall that $I$ is the minimal prime ideal of $A$ hence $I$ corresponds to the generic point of $U$.

 Let us denote $A_I$, the localization of $A$ at $I$ by $B$ and denote the ideal $IB$ by $J$. Moreover we denote $\Spec(B)$ by $V$ and similarly  denote $\Spec(B/J)$ by $V_{\red}$. Let $f: V\to U$, $f_{\red}: V_{\red}\to U_{\red}$, and $j: V_{\red}\to V$ be the natural maps.

Since $f: V\to U$ is flat, we can define the pull-back map $f^*: G_0(U)\to G_0(V)$.

Let us denote the class $[\mathcal{O}_V]$ in $G_0(V)$ by $e_V$. By definition $f^*(e_U)=e_V$. If $e_U=0$ then we have $e_V=0$ and $j_*^{-1}(e_V)=0$.

On the other hand $B/J=A_I/I_I\cong \Frac(A/I)$ is a field hence $G_0(V_{\red})=G_0(B/J)\cong \mathbb{Z}$. Similar to Equation (\ref{equa: decompose of e_U}) we have
$$
j_*^{-1}(e_V)=[B/J]+[J/J^2]+\ldots +[J^{l-1}/J^l]+[J^l].
$$
 Each of the $J^{m-1}/J^m$ is a vector space over the field $B/J$ hence the right hand side cannot be zero in $G_0(V_{\red})$.
\end{proof}

\begin{prop}\label{prop: n times reduced for general elements in G_0}
Let $U$ and $n$ be as in Lemma \ref{lemma: n times reduced in G_0}. Then for any element $a\in K_0(U)$ we have
$$
c(a)=n\,i_*c\,i^*(a),
$$
i.e. the diagram
$$
\begin{CD}
K_0(U) @>c>> G_0(U)\\
@VVn\,i^*V @A\cong Ai_*A\\
K_0(U_{\red}) @>c>> G_0(U_{\red})
\end{CD}
$$
commutes.
\end{prop}
\begin{proof}
We need the following lemma.

\begin{lemma}\label{lemma: i_* is a module map}
For any Noetherian scheme $U$, $G_0(U_{\red})$ has a $K_0(U)$-module structure. Moreover, the map $i_*: G_0(U_{\red})\to G_0(U)$ is a morphism of $K_0(U)$-modules.
\end{lemma}
\begin{proof}[Proof of Lemma \ref{lemma: i_* is a module map}]
First  the $K_0(U)$-module structure on $G_0(U_{\red})$ is given by composing with $i^*$. More explicitly, for $a\in K_0(U)$ and $m\in G_0(U_{\red})$ we define
$$
a\cdot m= i^*(a)\cdot m
$$
where the right hand side uses the $K_0(U_{\red})$-module structure on $G_0(U_{\red})$.

Then we need to show that $i_*$ is a $K_0(U)$-module map, i.e.
$$
i_*(i^*(a)\cdot m)=a\cdot i_*(m).
$$
But this is exactly the projection formula.
\end{proof}

Now we can prove Proposition \ref{prop: n times reduced for general elements in G_0}. Let us denote $[\mathcal{O}_U]\in K_0(U)$ by $1_U$ and $[\mathcal{O}_{U_{\red}}]\in K_0(U_{\red})$ by $1_{U_{\red}}$. Then it is clear that
$$
c(1_U)= e_U \text{ and } c(1_{U_{\red}})=e_{U_{\red}}.
$$
Then for any $a\in K_0(U)$ we have
\begin{equation*}
\begin{split}
c(a)=&c(a\cdot 1_U)\\
=&a\cdot e_U\\
=&a\cdot (n i_*(e_{U_{\red}})) (\text{ Lemma } \ref{lemma: n times reduced in G_0})\\
=&n(a\cdot  i_*(e_{U_{\red}}))\\
=&ni_*(i^*(a)\cdot e_{U_{\red}}) (\text{ Lemma } \ref{lemma: i_* is a module map})\\
=&n\,i_*c\,i^*(a).
\end{split}
\end{equation*}
\end{proof}

Now we are ready to prove the following Proposition, which is the general version of Proposition \ref{prop: Cartan homo has infinite image, reduced case}.

\begin{prop}\label{prop: Cartan homo has infinite image, general case}
Let $X$ be a   projective curves of geometric genus $g\geq 1$ over an algebraically closed field $k$, then the image of the Cartan homomorphism
$$
c: K_0(X)\to G_0(X)
$$
is not finitely generated.
\end{prop}
\begin{proof}

First let $U$ be as in Lemma \ref{lemma: n times reduced in G_0} and Proposition \ref{prop: n times reduced for general elements in G_0}. By Proposition \ref{prop: n times reduced for general elements in G_0} and Theorem \ref{thm: devissage}there is a non-zero integer $n$ such that the following diagram commutes
$$
\begin{CD}
K_0(U) @>c>> G_0(U)\\
@VVn\,i^*V @VV(i_*)^{-1}V\\
K_0(U_{\red}) @>c>> G_0(U_{\red})
\end{CD}
$$
hence the diagram
$$
\begin{CD}
K_0(X) @>c>> G_0(X)\\
@VrVV @VVrV\\
K_0(U) @>c>> G_0(U)\\
@VVn\,i^*V @VV(i_*)^{-1}V\\
K_0(U_{\red}) @>c>> G_0(U_{\red})
\end{CD}
$$
commutes. For short we have
$$
\begin{CD}
K_0(X) @>c>> G_0(X)\\
@VVn\,i^*  rV @VV(i_*)^{-1} rV\\
K_0(U_{\red}) @>c>> G_0(U_{\red})
\end{CD}
$$

Now assume the image of $c: K_0(X) \to G_0(X)$ is finitely generated. Since $U_{\red}$ is smooth, the $c: K_0(U_{\red}) \to G_0(U_{\red})$ in the above diagram is an isomorphism, hence the image of $n\,i^* r$ is also finitely generated.

Next we observe that we have the commutative diagrams
$$
\begin{CD}
K_0(X) @>n\,i^*>> K_0(X_{\red})\\
@VVrV @VV rV\\
K_0(U) @>n\,i^*>> G_0(U_{\red})
\end{CD}
$$
and
$$
\begin{CD}
K_0(X) @>\det>> \Pic(X)\\
@Vn\,i^*VV @VVn\,i^*V\\
K_0(X_{\red}) @>\det>> \Pic(X_{\red})\\
@VVrV @VVrV\\
K_0(U_{\red}) @>\det>> \Pic(U_{\red})
\end{CD}
$$

From the left-bottom composition of the above diagram we know that the image of $\det\circ r\circ (n\,i^*)$ is finitely generated.

On the other hand we will study the top-right composition of the above diagram. By Proposition \ref{prop: K_0 to Pic det} the map det is   surjective and by Lemma \ref{lemma: ext of line bundles on singular curves} the map $r$ is also surjective. As for the map $i^*$ we need the following lemma.

\begin{lemma}\label{lemma: Pic is surj when pull back to reduced}[\cite{liu2002algebraic} Lemma 7.5.11]
Let $X$ be a connected projective curve over an algebraically closed field $k$, Then $i^*:  \Pic(X)\to \Pic(X_{\red})$ is surjective.
\end{lemma}
\begin{proof}[Proof of Lemma \ref{lemma: Pic is surj when pull back to reduced}] See \cite{liu2002algebraic} Lemma 7.5.11.
\end{proof}

Then it is clear that the image of $r\circ(ni^*)\circ \det$ is $n\Pic(U_{\red})$. Compare with the left-bottom composition we get the conclusion that $n\Pic(U_{\red})$ is finitely generated, which is contradictory to Lemma \ref{lemma: Picard gp of U is infi generated}.
\end{proof}

\begin{coro}\label{coro: non-exist of full excep collection for factor through general case}
Let $X$ be a  projective curves of geometric genus $g\geq 1$ over an algebraically closed field $k$. If the inclusion $\rD^{\perf}(X)\to \rD^b(\coh(X))$ factors through a triangulated category $\mathcal{S}$, then $\mathcal{S}$ cannot have a full exceptional collection.
\end{coro}
\begin{proof}
The proof is almost the same as that of Corollary \ref{coro: non-exist of full excep collection for factor through reduced case} except that we use Proposition \ref{prop: Cartan homo has infinite image, general case} instead of Proposition \ref{prop: Cartan homo has infinite image, reduced case}.
\end{proof}

\begin{thm}\label{thm: non-exist of full excep collection general case}[See Theorem \ref{thm: non-exist of full excep collection in the introduction}]
Let $X$ be a projective curve over an algebraically closed field $k$. Let $(\mathscr{T},\pi^*,\pi_*)$ be a categorical resolution of $X$. If the geometric genus of $X$ is $\geq 1$, then $\mathscr{T}^c$ cannot have a full exceptional collection.

In other words, $X$ has a categorical resolution which admits a full exceptional collection if and only if the geometric genus of $X$ equals to $0$.
\end{thm}
\begin{proof}
Since we have Proposition \ref{prop: g=0 has a full exc coll}, it is sufficient to prove the first claim of the theorem, which is a direct consequence of Corollary  \ref{coro: non-exist of full excep collection for factor through general case}.
\end{proof}

\begin{rmk}\label{rmk: haven't use T is smooth}
In the proof we did not use the fact the $\mathscr{T}$ is a smooth triangulated category.
\end{rmk}

\begin{rmk}\label{rmk: the proof does not work for non-algebraically closed field}
The proof of Theorem \ref{thm: non-exist of full excep collection general case} fails if the base field $k$ is not algebraically closed. The main reason is when $k$ is not algebraically closed, the picard group may be finitely generated. See Remark \ref{rmk: picard group non-algebraic closed case} after Corollary \ref{coro: Picard is infinitely generated}.

Nevertheless, we expect that the result of Theorem \ref{thm: non-exist of full excep collection general case} is still true in the non-algebraically closed case. We believe  that a proof could be achieved through a systematic study of the behavior of categorical resolution under scalar extension and we will leave this topic for a future paper.
\end{rmk}

It is worthwhile  to mention that we have another application of Proposition \ref{prop: Cartan homo has infinite image, general case} (thanks to Igor Burban for pointing it out).

\begin{thm}\label{thm: non-exist of f.d algebra as a cat resolution}
Let $X$ be a projective curve over an algebraically closed field $k$ of geometric genus $\geq 1$. Let $(\mathscr{T},\pi^*,\pi_*)$ be a categorical resolution of $X$. Then $\mathscr{T}^c$ cannot have a tilting object, moreover there cannot be a finite dimensional $k$-algebra $\Lambda$ of finite global dimension such that
$$
\mathscr{T}^c\cong \rD^b(\Lambda-\text{mod})
$$
where $\rD^b(\Lambda-\text{mod})$ is the derived category of bounded complexes of finitely generated $\Lambda$-modules.
\end{thm}
\begin{proof}
With Proposition \ref{prop: Cartan homo has infinite image, general case} it is sufficient to prove that the Grothendieck group $K_0(\rD^b(\Lambda-\text{mod}))$ is finitely generated. The proof is  as follows: Since $\Lambda$ is finite dimensional, it is a finitely generated Artinian $k$-algebra, hence every finitely generated $\Lambda$-module has a composition series. Moreover the set of isomorphic classes of simple $\Lambda$-module is finite. We get the desired result.
\end{proof}

\begin{rmk}
Again in the proof we did not use that fact that $\Lambda$ is of finite global dimension, which corresponds to the smoothness of $\mathscr{T}$ .
\end{rmk}

\section*{Acknowledgement}
The author first wants to thank Valery Lunts for introducing him to this topic and for very helpful comments and suggestions during this work. Igor Burban suggests the author to investigate the non-irreducible case and shares the ideas on tilting object and the author is grateful to him too. Moreover the author would like to thank Dave Anderson,
K\c{e}stutis \v{C}esnavi\v{c}ius , Georges Elencwajg, Volodimir Gavran, Adeel Khan, S\'{a}ndor Kov\'{a}cs, and Jason Starr for their help on algebraic K-theory and the theory of algebraic curves.

\bibliography{fullexcepcurvesbib}{}
\bibliographystyle{plain}

\end{document}